\documentclass[12pt,a4paper]{article}
\usepackage[top=2.5cm,bottom=2.5cm,left=2.5cm,right=2.5cm]{geometry}

\usepackage{amssymb,amsmath,graphicx,color,amsthm,centernot}
\usepackage[capitalise]{cleveref}
\usepackage[labelformat=simple]{subcaption}
\usepackage[font=small,labelfont=bf,width=12.5cm]{caption}
\usepackage{booktabs}

\usepackage{multirow,tabularx}

\newtheorem{theorem}{Theorem}[section]
\newtheorem{lemma}[theorem]{Lemma}
\newtheorem{corollary}[theorem]{Corollary}
\newtheorem{proposition}[theorem]{Proposition}
\newtheorem{observation}{Observation}
\newtheorem{conjecture}[theorem]{Conjecture}

\newtheorem{question}[theorem]{Question}

    {\noindent \emph{Proof.} {}{#1}{}}{$~$\hfill $~\blacklozenge$ \vspace{0.2cm}}

\definecolor{defblue}{rgb}{0.4,0,0.84}
\definecolor{greyblue}{rgb}{0.23,0.4,0.70}
\definecolor{orange}{rgb}{1.0,0.5,0.2}
\definecolor{violet}{rgb}{0.55,0,0.55}

\usepackage{url}
\makeatletter
\g@addto@macro{\UrlBreaks}{\UrlOrds}
\makeatother

\newcolumntype{Y}{>{\centering\arraybackslash}X}

\begin{document}

\title{{\bf Improper coloring of toroidal graphs}}

\author
{
	Alexandra~Kola\v ckovsk\'a \quad
	M\'aria~Macekov\'a \\
	Roman~Sot\'ak\thanks{Pavol Jozef \v Saf\'{a}rik University, Faculty of Science, Ko\v{s}ice, Slovakia.} \quad
	Diana~\v{S}vecov\'{a}\footnotemark[1] \quad
}

\maketitle

{
\begin{abstract}
A graph $G$ is called $(d_1,\dots,d_k)$-colorable if its vertices can be partitioned into $k$ sets $V_1,\dots,V_k$ such that $\Delta(\langle V_i\rangle_G)\leq d_i, i\in \{1,\dots, k\}$.
If $d_1 = \dots = d_k = m$ we say that $G$ is $k$-colorable with defect $m$.
A coloring with at least one $d_i, i\in \{1,\dots, k\}$, greater than $0$ is called an improper coloring.
It is known that toroidal graphs are properly $7$-colorable, therefore they are $7$-colorable with defect $0$.
It was also proved that toroidal graphs are $5$-colorable with defect $1$ and $3$-colorable with defect $2$.
The question whether they are $4$-colorable with defect $1$ remains open.

In this paper we focus on improper coloring of toroidal graphs with values of defects being not all equal.
We prove that these graphs are $(0,0,0,0,0,1^*)$-colorable, $(0,0,0,0,2)$-colorable and $(0,0,0,1^*,1^*)$-colorable (a star means that there is an improper coloring in which subgraph induced by the corresponding color class contains at most one edge).
Choi and Esperet in [Improper coloring of graphs on surfaces, J. Graph Theory $91(1)\,(2019), 16-34$] proved that every graph of Euler genus $eg > 0$ is $(0, 0, 0, 9eg - 4)$-colorable.
From this result it follows that toroidal graphs are $(0,0,0,14)$-colorable.
We decreased the value $14$ and proved that toroidal graphs are $(0,0,0,4)$-colorable.

We also show that all 6-regular toroidal graphs except $K_7$ and $T_{11}$ are $(0,0,0,1)$-colorable.
Finally, we discuss the colorability of graphs embeddable on $N_1$ and show that they are $(0,0,0,2)$-colorable.
\end{abstract}
}


\section{Introduction}
An \textit{improper coloring} of a graph $G$ (also known as \textit{defective coloring}) is a vertex coloring which allows adjacent vertices to receive the same color.
We say that $G$ is \textit{$(d_1, \dots , d_k )$-colorable} if each color class $V_i$ induces a subgraph of a maximum degree at most $d_i$, where $d_i$ is called a \textit{defect} of color class $V_i,\ i\in \{1,\dots, k\}$.
If all defects are equal to $m\in \mathbb{N}$, we say that $G$ is \textit{$k$-colorable with defect $m$}.
Note that proper vertex coloring has defects of all color classes being 0.

The Four Color Theorem~\cite{AppHak} implies that planar graphs are $(0,0,0,0)$-colorable.
When we decrease the number of colors from four to three, Cowen, Cowen, and Woodall~\cite{Cow} proved that every planar graph is $(2,2,2)$-colorable.
This result is the best possible as for each $k,\ k\geq 1$, there exists a planar graph which is not $(1,k,k)$-colorable (construction of such graph can be found in~\cite{ChoEsp}).
Furthermore, for any integer $k$ it is not difficult to construct a planar graph that is not $(k, k)$‐colorable (see~\cite{Cow}).

Many results have been obtained also for planar graphs with prescribed girth or bounded maximum average degree.
\v Skrekovski~\cite{Skr} proved that for any $k$ there is a triangle-free planar graph that is not $(k, k)$-colorable. 
%That shows that even in this class of graphs it is not possible to achieve results for improper coloring using two colors.
However, planar graphs with girth at least five are $(k_1, k_2)$-colorable for some positive integers $k_1, k_2$ (see for example~\cite{BorKosYan, ChoChoJeo, ChoRas}).
Note that some results were originally proved for graphs with bounded maximum average degree, but they hold also for planar graphs with girth at least $g$ (as they have maximum average degree less than $\frac{2g}{g-2}$, see~\cite{BorKosNes}). 

When trying to generalize the results regarding coloring of planar graphs, natural way is to take a closer look at coloring of graphs embeddable on surfaces of higher genus.
Since any toroidal graph is properly 7-colorable (see~\cite{Hea}), it is 7-colorable with defect~0.
Cowen and Goddard in~\cite{CowGod} proved that every toroidal graph is $(1,1,1,1,1)$- and (similarly as planar graphs) $(2,2,2)$-colorable. 
Question whether all toroidal graphs are $(1,1,1,1)$-colorable remains open (proposed in~\cite{CowGod}).

Regarding improper coloring of graphs embeddable on a surface of Euler genus $eg$, Choi and Esperet~\cite{ChoEsp} proved the following:
\begin{theorem}
\label{ChoEsp}
For each $eg > 0$, every graph embedded on a surface of Euler genus $eg$ is $(0, 0, 0, 9eg - 4)$-colorable and $(2, 2, 9eg - 4)$-colorable.
\end{theorem}

In~\cite{ChoEsp} it was also shown that $9eg-4$ cannot be replaced by a sublinear function of~$eg$.
They also showed that every triangle-free graph embeddable on a surface of Euler genus $eg$ is $(0, 0,K_3(eg))$-colorable, $K_3(eg) = \lceil\frac{10eg+32}{3}\rceil$, where $K_3(eg)$ cannot be replaced by a sublinear function of $eg$.
Choi et al.~\cite{ChoChoJeo} proved that every graph with girth at least $5$ embeddable on a surface of Euler genus $eg$ is $(1,K_4(eg))$-colorable, $K_4(eg) = \max\{10, \lceil\frac{12eg+47}{7}\rceil\}$.
Finally, every graph of girth at least $7$ embeddable on a surface of Euler genus $eg$ is $(0, 5 +\lceil\sqrt{14eg + 22}\rceil)$-colorable (see~\cite{ChoEsp}). 
Authors in~\cite{ChoEsp} also showed that there is a constant $c > 0$ such that for infinitely many values of $eg$ there are graphs with girth at least 7 embeddable on a surface of Euler genus $eg$ with no $(0, \lfloor c\sqrt{eg}\rfloor)$-coloring.

In this paper we consider improper colorings where $d_i \neq d_j$ for at least one combination of $1 \leq i,j \leq k$.
%We proved tight results for improper coloring of toroidal graphs with six and five colors.
We prove that every toroidal graph is $(0,0,0,0,0,1)$-, $(0,0,0,0,2)$- and $(0,0,0,1,1)$-colorable and show some additional improvements of these results. 

Theorem \ref{ChoEsp} implies that every toroidal graph is $(0, 0, 0, 14)$-colorable.
Last value in the list of defects can be lowered and we show that every toroidal graph is $(0, 0, 0, 4)$-colorable.
It is easy to see that not every toroidal graph admits a $(0, 0, 0, 2)$-coloring (as $K_7$ is not $(0, 0, 0, 2)$-colorable).
However, it remains an open question whether or not every toroidal graph is $(0, 0, 0, 3)$-colorable.
We prove this to be true for toroidal graphs fulfilling some additional conditions.

For 6-regular toroidal graphs we show that they are $(0,0,0,1)$-colorable with only two exceptions being $K_7$ and a special triangulation $T_{11}$.

Finally, we prove that every graph embeddable on $N_1$ is $(0,0,0,2)$-colorable. 

\section{Basic notations}

In this paper we use a standard graph theory terminology according to books~\cite{BonMur} and~\cite{GroTuc}.
However, we introduce some more frequent terms.

All graphs in this paper are connected and simple, i.e. without loops and multiple edges.
Let $G$ be such a graph.
The \textit{degree} $\deg(v)$ of a vertex $v$ equals to the number of edges incident with $v$ in $G$. 
If $\deg(v)=k$ ($\deg(v)\le k, \deg(v)\ge k$), we say that $v$ is a \textit{$k$-vertex} (\textit{$k^-$-vertex}, \textit{$k^+$-vertex}).
We denote the set of vertices adjacent to $v$ as $N(v)$.
\textit{Girth} $g(G)$ of $G$ denotes the length of the shortest cycle in $G$.
\textit{Average degree} of $G$ is defined as $\mathrm{ad}(G) = \frac{2|E(G)|}{|V(G)|}$ and \textit{maximum average degree} of $G$ is $\mathrm{mad}(G) = \max\limits_{H\subseteq G} \mathrm{ad}(H)$.
We denote a subgraph of $G$ induced on $V'\subseteq V(G)$ as $\langle V'\rangle_G$. 

Graph $G$ is \textit{(properly) $k$-colorable} if there exists an assignment of colors $\{1, \dots, k\}$ to the vertices of $G$ such that adjacent vertices receive different colors.

Graph $G$ is called $(d_1, \dots, d_k)$\textit{-colorable} if the vertex set of $G$ can be partitioned into $k$ sets $V_1, \dots, V_k$ such that $\Delta(\langle V_i\rangle_G) \leq d_i,\ i\in\{1,\dots, k\}$. 
If $d_1 = \dots = d_k = m$, then we say that $G$ is \textit{$k$-colorable with defect $m$}.
An edge $e=uv$ is called \textit{monochromatic} in coloring $\varphi$ if both $u$ and $v$ receive the same color.
We say that graph $G$ is \textit{almost $k$-colorable} if there is a coloring $\varphi$ of $G$ such that at most one edge in $\varphi$ is monochromatic.
If $G$ is $(d_1, \dots, d_{i-1}, 1, d_{i+1} \dots, d_k)$-colorable in such a way that $\langle V_i\rangle_G$ contains at most one edge, we will mark the corresponding defect $d_i = 1$ with a star in the upper index, i.e., $G$ is $(d_1, \dots, d_{i-1}, 1^*, d_{i+1}, \dots, d_k)$-colorable.
Trivially, the following observation holds:

\begin{observation}
\label{cykly}
Every cycle is $(0,1^*)$-colorable and $(2)$-colorable.
\end{observation}

Let $\varphi$ be a $k$-coloring of $G$.
In the paper we use a notation $C_\varphi(S) = \{\varphi(v) : v \in S\}$ for $S\subseteq V(G)$.

\textit{Surface} is a connected 2-manifold which is closed, i.e. it is compact and its boundary is empty.
We consider two types of surfaces $-$ \textit{orientable} and \textit{non-orientable}.
The orientable surface $\mathbb{S}_k$ is obtained from a sphere by adding $k$ handles.
The non-orientable surface $\mathbb{N}_k$ is formed by adding $k$ \textit{cross-caps} (M\"{o}bius bands) to a sphere.
The \textit{genus} of a surface is equal to $k$ for both $\mathbb{S}_k$ and $\mathbb{N}_k$.
The \textit{Euler genus} ${\rm eg}(\mathcal{S})$ of a surface $\mathcal{S}$ equals twice its genus if $\mathcal{S}$ is orientable, and is equal to its genus if $\mathcal{S}$ is non-orientable.
The \textit{embedding} of a graph $G$ on a surface is defined as a continuous one-to-one function from a topological representation of $G$ into the surface.
For a given embedding of $G$ on a surface $\mathcal{S}$, the components of $\mathcal{S}\setminus G$ are called \textit{regions}. 
If each region is homeomorphic to an open disk, the embedding is called \textit{cellular} and the regions are called \textit{faces}.
Further in the paper, by embedding of a graph we mean a cellular embedding of this graph.
By Euler genus ${\rm eg}(G)$ of a graph $G$ embedded into a surface $\mathcal{S}$ we mean ${\rm eg}(\mathcal{S})$. 
For a graph which is embedded into a surface, the \textit{Euler formula} represents the sum
\begin{equation} \label{EulerFormula}
 |V| - |E| + |F|,
\end{equation}
and is equal to $2 - {\rm eg}(G)$, which equals 0 for toroidal graphs. 
It follows that for toroidal graphs the number of faces equals $|E| - |V|$. 

Let $G$ be a graph embedded on some surface and let $f$ be a face of $G$.
Then a \textit{boundary walk} of $f$ is the shortest walk consisting of edges and vertices of $G$ as they are reached when walking along the whole boundary of $f$, following some orientation of $f$ in the usual topological sense. 
The \textit{degree} $\deg(f)$ of $f$ is the length of its boundary walk.
Face $f$ with $\deg(f) = k$ is called a \textit{$k$-face}.
The sum of the degrees of all faces of $G$ equals twice the number of its edges, i.e. $\sum_{f\in F} \deg(f) = 2|E|$. 
Clearly, every face has degree at least $3$. 
If in a particular embedding of $G$ all faces have degree exactly three, we call this embedding a \textit{triangulation}.

A cycle $C$ of $G$ is called \textit{contractible} if one of the regions bounded by $C$ is homeomorphic to an open disc, otherwise $C$ is \textit{non-contractible}. 
Furthermore, $C$ is called \textit{separating} if $C$ separates a surface into two connected regions such that both contain at least one vertex or edge of $G$; otherwise $C$ is \textit{non-separating}.

\textit{Join} of graphs $G_1$ and $G_2$ is defined as a graph $G_1\vee G_2$ with a vertex set $V (G_1 \vee G_2) = V (G_1) \cup V (G_2)$ and an edge set $E(G_1 \vee G_2) = E(G_1) \cup E(G_2) \cup \{uv : u \in V (G_1), v \in V (G_2)\}$.

We define two special toroidal graphs.
Graph $H_7$ can be obtained by Haj\'os construction applied to two copies of $K_4$.
Graph $T_{11}$ is a special toroidal triangulation with 11 vertices (see Figure \ref{fig:H7,T11}).

\begin{figure}[h!]
				$$
					\includegraphics[scale=1]{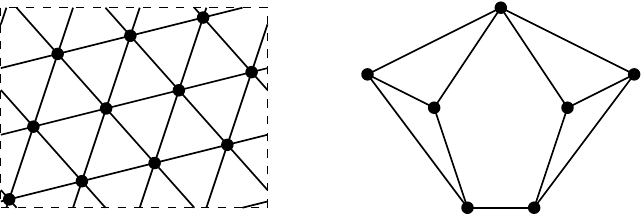}
				$$
				\caption{Graphs $T_{11}$ (left) and $H_7$ (right).}
				\label{fig:H7,T11}
\end{figure}	

Throughout the paper we will use the specific transformation of a toroidal graph $G$ into a planar graph $H$ described in~\cite{CowGodJes} that looks as follows.
Consider the shortest non-contractible cycle $C$ of $G$.
By cutting along $C$ we get two copies $C_1$ and $C_2$ of it, which we contract into two vertices.
Note that any multiedge that appears after contraction is replaced by a single edge.
The resulting graph $H$ is embeddable on a sphere, hence, it is planar.
We will denote this transformation $T(G,C)$.

\section{Improper coloring with 6 colors}
\label{6col}

Concerning proper coloring of toroidal graphs, there are graphs that need at least seven colors (as $K_7$ needs exactly seven colors).
If we want to decrease the number of colors to six, we need to increase at least one defect-value in general.
%at least one monochromatic edge will occur.

\begin{theorem}
Every toroidal graph is $(0,0,0,0,0,1^*)$-colorable.
\end{theorem}

\begin{proof}
%We may assume that $G$ is a maximal toroidal graph, and 
We transform toroidal graph $G$ into a planar graph $H$ using the transformation $T(G,C)$, where $C$ is a shortest non-contractible cycle in $G$. 
Let $u$ and $v$ be the vertices that represent the contracted copies of $C$.
Graph $H$ is planar, therefore it is properly $4$-colorable. 
Let $\varphi$ be a proper 4-coloring of $H$.
We define $(0,0,0,0,0,1^*)$-coloring $\varphi'$ of $G$ in the folowing way: for all vertices $x\in V(G\setminus C)$ let $\varphi'(x) = \varphi(x)$.
For $x\in V(C)$ we will use colors $5$ and $6$, using Observation \ref{cykly}. 
If $\varphi'$ contains a monochromatic edge, let $6$ be the color used on its end-vertices.
%Vertices colored with $1,2,3,4,5$ induce independent sets and there are at most two adjacent vertices colored with $6$.
Hence it follows that $G$ is $(0,0,0,0,0,1^*)$-colorable.
As the complete graph $K_7$ is not properly $6$-colorable, the result is the best possible.
\end{proof}

\section{Improper coloring with 5 colors}
\label{5col}

Considering improper colorings of toroidal graphs that use five colors, it is not difficult to see that $K_7$ is not $(0,0,0,0,1)$-colorable (otherwise there are three vertices colored with color 5 inducing $C_3$ in $K_7$).
Therefore, Theorem \ref{00002} is the best possible as the value 2 in the list cannot be decreased.

\begin{theorem} \label{00002}
Every toroidal graph is $(0,0,0,0,2)$-colorable.
\end{theorem}

\begin{proof}
Let $G$ be a toroidal graph, $C$ be a shortest non-contractible cycle in $G$. 
We again use the transformation $T(G,C)$ for $G$ and let $H$ be a resulting planar graph in which $u, v$ correspond to $C$. 
Let $\varphi$ be an arbitrary 4-coloring of $H$.
We can color vertices of $G$ in such a way that all vertices from $\langle H\setminus \{u,v\}\rangle_G$ receive the same color as in $\varphi$ and remaining vertices (i.e. vertices belonging to cycle $C$) receive a new color 5. 
Then the resulting coloring is a $(0,0,0,0,2)$-coloring of $G$ (using Observation \ref{obs} and the fact that $C$ is chordless).
\end{proof}

Natural question is whether we can lower the value 2 in Theorem \ref{00002} in case we allow more than one non-zero defect.
The answer is positive, but we need the following result:

\begin{theorem}[Thomassen~\cite{Tho}]
\label{thomassen}
Let $G$ be a toroidal graph. Then $G$ is $5$-colorable if and only if $G$ does not contain $K_6$ or $C_3\vee C_5$ or $K_2\vee H_7$ or $T_{11}$.
\end{theorem}

The proof of Theorem \ref{thomassen} uses the following proposition that we use in proofs of Theorem \ref{00001} and Theorem \ref{00011}.

\begin{proposition}[Thomassen~\cite{Tho}]
\label{thomassen.proposition}
Let $G$ be a planar graph with outer cycle $S = \{x_1, x_2,$ $\dots, x_k, x_1\}, k \leq 6$. Let $\varphi$ be a $5$-coloring of $S$. 
Then $\varphi$ can be extended to a $5$-coloring of $G$ if and only if none of $(i), (ii), (iii)$ below holds:
\begin{enumerate}
\item[(i)] $|C_{\varphi}(S)| = 5$ and $G$ has a vertex $v$ adjacent to the vertices of every color of $C_{\varphi}(S)$.
\item[(ii)] $k = 6,\ |C_{\varphi}(S)| = 4,\ G \setminus S$ contains two adjacent vertices $u, v$, each of them joined to the vertices of all four colors of $C_{\varphi}(S)$.
\item[(iii)] $k = 6,\ |C_{\varphi}(S)| = 3,\ G \setminus S$ contains three pairwise adjacent vertices $u, v, w$, each of them is joined to the vertices of all three colors of $C_{\varphi}(S)$.
\end{enumerate}

\end{proposition}

\begin{theorem}
\label{00001}
Every toroidal graph not containing $K_7$ as a subgraph is $(0,0,0,0,1^*)$-colorable.
\end{theorem}

\begin{proof}
As by Theorem~\ref{thomassen} every graph not containing $H \in Z = \{K_6, C_3\vee C_5, K_2\vee H_7, T_{11}\}$ is $(0,0,0,0,0)$-colorable, it suffices to show that any graph $G$ containing $H \in Z$ as a subgraph and not containing $K_7$ is $(0,0,0,0,1^*)$-colorable.

First we show that any embedding of graph from $Z\setminus\{K_6\}$ into torus has maximum face degree at most four.

For embedding of $C_3\vee C_5$ it holds that $|V|=8,\ |E|= 3+5+3\cdot5 = 23$, therefore $|F| = 15$ (using Euler's formula (\ref{EulerFormula})).
It follows that $45 = 3\cdot15 \leq \sum_{f\in F} \deg(f) = 2|E| = 2\cdot23 = 46$, which implies that in any embedding of $C_3 \vee C_5$ exactly one face has degree four, all the others have degree three.

Similarly, for $K_2 \vee H_7$ we have $|V| = 9,\ |E|=26,\ |F| = 17$.
Then $51 = 3\cdot17 \leq \sum_{f\in F} \deg(f) = 2|E| = 2\cdot26 = 52$ and again there is exactly one $4$-face in any toroidal embedding of $K_2 \vee H_7$.

Finally, $T_{11}$ is a triangulation, therefore in any toroidal embedding of this graph all faces have degree $3$.

Let $G$ be an arbitrary toroidal graph containing $H \in Z \setminus \{K_6\}$ as a subgraph and not containing $K_7$. 
The graph $H$ is $(0,0,0,0,1^*)$-colorable, since there exists such a coloring $\varphi$ of a graph $H$ (see the right side of Figure \ref{fig:(0,0,0,0,1)-coloring}).
Namely, in $C_3\vee C_5$ we color $C_3$ with three distinct colors and $C_5$ with two distinct colors.% in such a way that the monochromatic edge is incident with 3-faces.
In $K_2\vee H_7$ we color vertices of $H_7$ properly with four colors and vertices of $K_2$ with the fifth color.
%This monochromatic edge can be incident either with two 3-faces or with 3-face and 4-face.
%Lastly, $T_{11}$ is a triangulation and again monochromatic edge is incident with two $3$-faces in any $(0,0,0,0,1^*)$-coloring.
Recall that in each $(0,0,0,0,1^*)$-coloring of graph $H$ a monochromatic edge is incident with a $3$-face and a $4^-$-face.

\begin{figure}[h!]
				$$
					\includegraphics[scale=1.2]{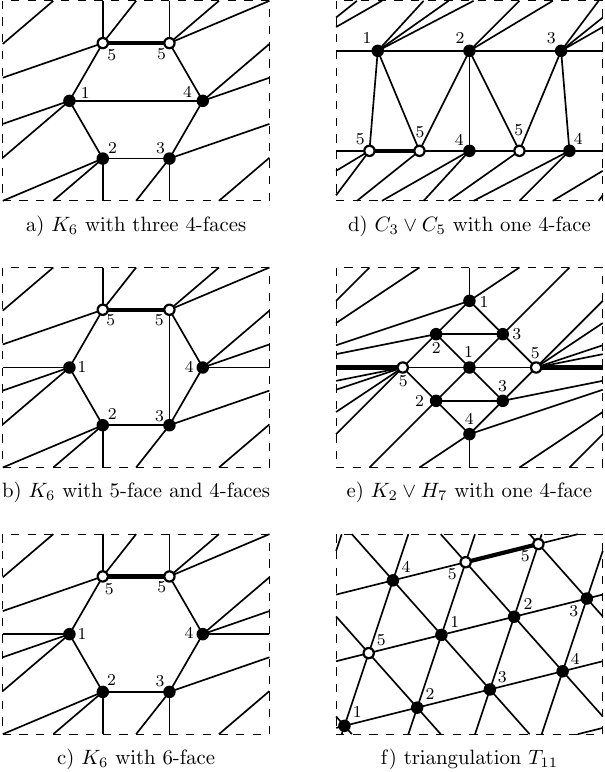}
				$$
				\caption{$(0,0,0,0,1^*)$-colorings of embedded graphs from $Z$. Only the fat edges are monochromatic.}
				\label{fig:(0,0,0,0,1)-coloring}
\end{figure}	

%In any embedding of $G$ we can apply this coloring on the corresponding vertices of $G$.
We can apply this $(0,0,0,0,1^*)$-coloring on the vertices that correspond to $H$ in $G$. Our aim is to show that this partial coloring $\varphi$ can be extended to the required coloring of the whole graph $G$.

%Now consider an arbitrary cycle $S$, which creates a boundary of a face $\alpha$ of the subgraph $H$.
%If $S$ contains a monochromatic edge, we will subdivide it by adding a $2$-vertex $v_S$ on it, which can be colored by one of the four remaining colors.
%The interior of any contractible cycle in $G$ induce a planar subgraph.
%It implies that, according to Proposition \ref{thomassen.proposition}, we can properly color the vertices inside of $S$ with five colors, as after possible subdivision of monochromatic edge of the subgraph $H$ the resulting toroidal embedding has all faces of degree at most five and no forbidden configuration appears.
%Thus, $G$ is $(0,0,0,0,1^*)$-colorable, where exactly one edge is monochromatic (as defined earlier).

Note that every vertex of $V(G) \setminus V(H)$ belongs to an interior of some face of $H$. First we subdivide the monochromatic edge $e$ by adding a 2-vertex $v_e$. The resulting graph $H'$ is 5-colorable and every face has length at most 5 (as $e$ was the only monochromatic edge in $\varphi$ and $v_e$ has four admissible colors). Moreover, if there is a 5-face $\alpha$, then it corresponds to a 4-face in $H$ and hence $|C_{\varphi}(V(\alpha))| \leq 4$. Then we can apply Proposition~\ref{thomassen.proposition} (as for arbitrary toroidal embedding of $G$ it holds that it is locally planar) to show that the interior of every face $\alpha \in F(H)$ can be properly colored with 5 colors. After removing of $v_e$ we obtain a desired $(0,0,0,0,1^*)$-coloring of $G$.

%Let $S$ be a facial-cycle of a $4^-$-face $\alpha$ of $H$. If $S$ contains a monochromatic edge $e$, we subdivide this edge by adding a 2-vertex $v_S$ on it. Then the resulting graph $H + v_S$ is 5-colorable and $S$ has length at most 5 (as $e$ was the only monochromatic edge in $H$ and $v_S$ has four admissible colors). Then we can apply Proposition~\ref{thomassen.proposition} (as for arbitrary toroidal embedding of $G$ it holds that it is locally planar) to show that the interior of $S$ can be properly colored with 5 colors. After removing of $v_S$ we obtain a desired $(0,0,0,0,1^*)$-coloring of $G$.

Next we discuss the case when $G$ contains $K_6$ as a subgraph.
For toroidal embedding of $K_6$ it holds that $|V| = 6$, $|E| = 15$, $|F| = 9$.
Therefore, as a consequence of Euler's formula we have $27 = 3\cdot9 \leq \sum_{f \in F} \deg(f) = 2|E| = 2\cdot15 = 30$, what gives us three possibilities for the toroidal embedding of $K_6$. 
Either we have six $3$-faces and three $4$-faces, or seven $3$-faces, one $4$-face and one $5$-face, or eight $3$-faces and one $6$-face (see the left side of Figure \ref{fig:(0,0,0,0,1)-coloring}).
It is easy to see that $K_6$ is $(0,0,0,0,1^*)$-colorable.
We choose a $(0,0,0,0,1^*)$-coloring $\varphi$ of $K_6$ in such a way that the pair of adjacent vertices receiving the same color is incident with the face of degree 5 or 6 (if the embedding does not contain such a face, otherwise choose $\varphi$ arbitrarily). 
Now we want to again use Proposition \ref{thomassen.proposition} to show that also $G$ is $(0,0,0,0,1^*)$-colorable.

Let $S$ be a facial-cycle of a $p$-face $\alpha$ of a subgraph $K_6$ in $G$. \\ %, where $p$ denotes a length of $S$.
If $p\leq 4$, we proceed in a similar way as in the previous case and subdivide a monochromatic edge if there is any.
Let $p'$ denote the degree of $\alpha$ after possible subdivision of the monochromatic edge in $K_6$.
Since $p'\leq 5$, we need to discuss only case $(i)$ from Proposition~\ref{thomassen.proposition}. %among the cases $(i), (ii), (iii)$ in Proposition \ref{thomassen.proposition}, it is only the case $(i)$ that can occur.
%In this case $p = 4$ and $S$ would contain a monochromatic edge, but then $|C(S)| \leq 4$, a contradiction.
In this case $|C_\varphi(S)| \leq 4$ (as $p' = 5$ only if $S$ contains a monochromatic edge) and hence it is possible to extend the coloring of $S$ to proper 5-coloring of its interior. \\
%Therefore, in the given $(0,0,0,0,1^*)$-coloring of the subgraph $K_6$ we can properly color the interior of the face $\alpha$ using $5$ colors.\\
If $5\leq p \leq 6$, then $S$ contains a monochromatic edge $e = uv$.
In this case we contract $e$ into a single vertex $u^*$ and $p' = p-1$.
Since again $p'\leq 5$, we need to investigate only case $(i)$ from Proposition~\ref{thomassen.proposition}.
%the cases $(ii)$ and $(iii)$ from Proposition \ref{thomassen.proposition} cannot occur. The only case to solve is the one where $p'=5$ and there exists a vertex $w$ which is adjacent to each of five vertices of the face $\alpha$ after the contraction of the edge $e$.
The only problematic case that needs to be discussed is when $p'=5$ (i.e. $p = 6$) and there is a vertex adjacent to all vertices of $S / e$ $-$ we show that this case cannot occur.
Every vertex $w$ can be adjacent to at most five vertices %of the original $6$-face of the embedding of $K_6$, for $G$ does not contain $K_7$ as a subgraph.
of $S$ (as $G$ does not contain $K_7$), let $t$ be a vertex of $S$ that is not adjacent to $w$.
%We color subgraph $K_6$ in such a way that $t$ is not incident with a monochromatic edge in this coloring ($t\neq u,\ t\neq v$).
Let $\varphi$ be such an initial coloring of $K_6$ in which $t$ is not incident with a monochromatic edge (i.e. $t\neq u,\ t\neq v$).
Then $t\neq u^*$ and after contraction of $e$ is $w$ adjacent only to at most four vertices of $S / e$. %the boundary of the face $\alpha'$.
Therefore also case $(i)$ is not feasible and we can extend $\varphi$ to a $(0,0,0,0,1^*)$-coloring of whole $G$. %it is possible to properly color the interior of each face with five colors.
\end{proof}

Values of the defects for the color classes in Theorem \ref{00001} are the best possible as graphs $C_3\vee C_5$, $K_2\vee H_7$ and $T_{11}$ are not $5$-colorable.
%When we allow also $K_7$ as a subgraph, then we obtain the following:
For toroidal graphs in general, we have the following result:

\begin{theorem}
\label{00011}
Every toroidal graph is $(0,0,0,1^*,1^*)$-colorable.
\end{theorem}

\begin{proof}
According to Theorem \ref{00001} it suffices to show that any toroidal graph containing $K_7$ as a subgraph is $(0,0,0,1^*,1^*)$-colorable.
Recall that $K_7$ is $(0,0,0,1^*,1^*)$-colorable.
Let $G$ be a toroidal graph that contains $K_7$ as a subgraph.
Since $K_7$ is a triangulation, we can properly color interior of every 3-face similarly as in proof of Theorem \ref{00001}, resulting into $(0,0,0,1^*,1^*)$-coloring of $G$.

%The case analysis is the same as in the last case of the previous proof except for the case $p = 6$, $p'=5$. 
%Then there exists a vertex $w$ in the interior of $S$ which is adjacent to all six vertices of cycle $S$.
%Let $t$ be a vertex of $S$ which is not incident with monochromatic edge $e$ and $e^*$ is an edge connecting $w$ and $t$.
%After contraction of $e$ and removal of $e^*$ none of the possibilities $(i), (ii), (iii)$ can occur, therefore we can properly color the interiors of all faces of subgraph $K_6$ with five colors.
%Since vertices $w$ and $t$ received the same color and at the same time have color different from that used on monochromatic edge $e$ in the proper $5$-coloring of the modified graph, the resulting coloring of $G$ is $(0,0,0,1^*,1^*)$-coloring. 
\end{proof}

%Adding an additional restriction concerning 3-faces of a graph, we gain another result.
%In the proof we use the following properties of toroidal graphs.
%
%\begin{lemma}\cite{XuZha}
%Let $G$ be a connected toroidal graph.
%Then one of the following holds:
%\begin{itemize}
%\item $G$ contains two adjacent 3-faces,
%\item $\delta(G)<4$,
%\item $G$ contains two adjacent 4-vertices,
%\item $G$ contains $(4,5,5)$-face.
%\end{itemize}
%\end{lemma}
%
%\begin{theorem}
%Every toroidal graph without adjacent triangles is $(0,0,0,0,0)$-colorable.
%\end{theorem}
%
%\begin{proof}
%Let $G$ be the minimal counterexample to the statement.
%Then $G$ is toroidal graph without adjacent triangles which is not properly 5-colorable, while an arbitrary subgraph $H$ of $G$ is 5-colorable.
%Since $G$  does not contain two adjacent triangles, it contains either vertex of degree at most 3, or two adjacent 4-vertices, or $(4,5,5)$-face.
%
%Suppose that $\delta<5$ and let $v$ be $4^-$ vertex in $G$.
%Then $G - v$ is $(0,0,0,0,0)$-colorable, and since $v$ had at most 4 neighbors in $G$, there is at least one remaining color by which $v$ can be colored.
%The resulting coloring is $(0,0,0,0,0)$-coloring of $G$, a contradiction. 
%Hence, all vertices in $G$ have degree at least 5.
%
%It follows that $G$ does not contain neither two adjacent 4-vertices, nor $(4,5,5)$-face, a contradiction.
%\end{proof}

\section{Improper coloring with 4 colors}
\label{4col}

If we consider improper 4-coloring with defect $d$, the question whether every toroidal graph is $(1,1,1,1)$-colorable remains open (see~\cite{CowGod}).
Xu and Zhang~\cite{XuZha} proved this to be true for the class of toroidal graphs without adjacent triangles.
To prove even stronger result (where we lower some of the values of defects) only one from the values of defects can be lowered to zero (as $K_7$ is not $(0,0,1,1)$-colorable).

Theorem \ref{ChoEsp} deals with one non-zero defect, and it yields the following corollary.

\begin{corollary}
Every toroidal graph is $(0,0,0,14)$-colorable.
\end{corollary}

This corollary gives an upper bound for the value of the non-zero defect $k$.
We have found a better upper bound for $k$ by adjusting the result of Choi and Esperet~\cite{ChoEsp}.

\begin{theorem}\label{0004}
Every toroidal graph is $(0,0,0,4)$-colorable.
\end{theorem}

\begin{proof}
Given a toroidal graph $G$, let $C$ be the shortest non-contractible cycle.
Let $G^* = T(G,C)$, and let $u$ and $v$ be the vertices that correspond to the contracted copies of $C$.
Let $P^*$ be the shortest $uv$-path in $G^*$, and let $P=P^*\setminus\{u,v\}$.
Now, contract all edges of $P^*$ into a single vertex $v^*$ to obtain a graph $G'$ that is planar, and therefore $4$-colorable.
Let $\varphi'$ be a proper 4-coloring of $G'$.
We define coloring $\varphi$ of original graph $G$ in the following way: for vertices $x\in V(C) \cup V(P)$ let $\varphi(x) = \varphi'(v^*)$, for all other vertices $y$ let $\varphi(y) = \varphi'(y)$.
Denote the induced subgraph consisting of vertices of $V(C) \cup V(P)$ as $H$.
We will show that $\Delta(H)\le4$.

Let $C = c_1 \dots c_kc_1$ and $P = p_1 \dots p_{\ell}$.
As $P^*$ is the shortest $uv$-path in $G^*$, only vertices $p_1$ and $p_\ell$ are adjacent to vertices of $C$.
Moreover, each vertex of $P$ has at most two neighbors in $P$.

We will use the following observation of Choi and Esperet~\cite{ChoEsp}.
\begin{observation}
\label{obs}
Let $G$ be a graph embedded on some surface. 
If $C$ is a shortest non-contractible cycle in $G$, then $C$ is an induced cycle and each vertex of $G$ has at most $3$ neighbors in $C$.
\end{observation}

From Observation \ref{obs} it follows that each vertex of $G$, including all vertices of $P$, has at most three neighbors in $C$.
Therefore, vertices $p_1$ and $p_{\ell}$ have degree at most four.

As every vertex of $C$ is in $H$ adjacent to two vertices of $C$ (since $C$ is induced) and at most two vertices of $P$ ($p_1$ and $p_{\ell}$), $H$ has maximum degree at most four.
\end{proof}

As $K_7$ is not $(0,0,0,2)$-colorable, from Theorem \ref{0004} it follows that every toroidal graph is $(0,0,0,k)$-colorable, where $k$ is either $3$ or $4$.
The question whether toroidal graphs are $(0,0,0,3)$-colorable in general remains open.

However, if we admit more than one non-zero defect, we can prove the following:
\begin{lemma}
Every toroidal graph is $(0,1,2,2)$-colorable.
\end{lemma}

\begin{proof}
Let $G$ be an arbitrary toroidal graph. Then $G$ is $(2,2,2)$-colorable (by Theorem \ref{ChoEsp}). Let $\varphi$ be a $(2,2,2)$-coloring of $G$.
As every color class in $\varphi$ induces the set of disjoint paths and cycles, let $H$ be a subgraph induced by one of these colors.
Then $H$ is $(0,1)$-colorable; let $\psi$ be a $(0,1)$-coloring of $H$ with two new colors (different from these used in $\varphi$).
We construct new coloring $\varphi'$ of $G$ in the following way:
for all vertices $x\in G - H$, let $\varphi'(x) = \varphi(x)$.
For all vertices $x\in H$, let $ \varphi'(x)=\psi(x)$.
Hence, $G$ is $(0,1,2,2)$-colorable.
\end{proof}

It remains open whether non-zero defects in the previous result can be lowered.

\subsection{Improper coloring of 6-regular graphs}

In order to prove the $(0,0,0,k)$-colorability with $k \le 3$ we deal with some special subclasses of toroidal graphs. We focused mainly on 6-regular toroidal graphs.

\bigskip

First, we introduce a characterization of $6$-regular graphs embedded in torus given by Altshuler~\cite{Alt}.
Given a rectangular grid of $m$ rows and $n$ columns, we label as $(i,j)$ the vertex on the $i$th row and $j$th column.
A \textit{$6$-regular right-diagonal grid $G[m\times n]$} is a graph with vertex set $V = \{(i, j): 1\leq i\leq m, 1\leq j\leq n\}$, where the vertex $(i, j)$ has neighbors
$(i, j-1), (i, j+1), (i-1, j), (i+1, j), (i+1, j-1), (i-1, j+1)$.
In the first coordinate we count modulo $m$ and in the second coordinate modulo $n$.

A \textit{$6$-regular right-diagonal shifted grid $G[m\times n, k],\ 1\leq k \leq m,\ k\in \mathbb{Z},$} is almost the same as $G[m \times n]$, but with added rotation before the vertices of the $n$th column are joined to the first column. 
In other words, the vertex $(i, n)$ is now adjacent to $(i +k-2, 1)$ and $(i +k-1, 1)$ and remains adjacent to $(i+1, n), (i+1, n-1), (i, n-1), (i-1, n),\ i = 1, 2,\dots, m$. 
$G[m \times n]$ is the same as $G[m \times n, 1]$ for any positive integers $m, n$. \\
Obviously, $G[m\times n, k]$ is a $6$-regular toroidal graph. 
%for any positive integers $m,n$ and $1\leq k \leq m$.
Following theorem shows that the converse is also true.

\begin{theorem}[Altshuler~\cite{Alt}]
\label{Alt}
Every $6$-regular toroidal graph is isomorphic to a $6$-regular right-diagonal shifted grid $G[m \times n, k]$ for some integers $m, n, k$.
\end{theorem}

We define a \textit{circulant graph} $G_n[S]$ with $S\subseteq\{1,2,\dots, \lfloor n/2\rfloor\}$ as a graph with $V = \{0,1,\dots,n-1\}$ and $E = \{ij: i-j \equiv x \mod n, x \in S\}$.
The set $S$ is called the \textit{generating set} of $G_n[S]$.

It holds that every $6$-regular shifted grid $G[m \times 1, i]$ is a circulant graph $G_m[1, i-2, i-1]$ (see~\cite{ColHut}).

It was conjectured (\'Ad\'am~\cite{Adam}) that $G_n[S] \cong G_n[S']$ if and only if $S=pS'$ for some unit $p \mod n$.
Alspach and Parsons~\cite{AlsPar} gave sufficient conditions for which $n$ the straight implication holds.
The opposite direction is easy to see, that is if $S=pS'$ for some unit $p \mod n$, then $G_n[S] \cong G_n[S']$.
%This fact will be important later.

Yeh and Zhu~\cite{YehZhu} characterized all $4$-colorable $6$-regular toroidal graphs.
%The $4$-colorability of $6$-regular toroidal graphs has already been characterized by Yeh and Zhu~\cite{YehZhu}.
%The following theorem holds:

\begin{theorem}\cite{YehZhu}
\label{YehZhu}
$6$-regular toroidal graphs are $4$-colorable, with the following exceptions:
\begin{enumerate}
\item $G \in \{G[3 \times 3, 2], G[3 \times 3, 3], G[5 \times 3, 2], G[5 \times 3, 3], G[5 \times 5, 3], G[5 \times 5, 4]\}$;
\item $G = G[m \times 2, 1]$ with $m$ odd;
\item $G = G_n[1, r, r+1]$ and $n = 2r+2, 2r+3, 3r+1$ or $3r+2$ and n is not divisible by 4;
\item $G = G_n[1,2,3]$ and $n$ is not divisible by 4;
\item $G = G_n[1, r, r+1]$ and $(r,n) \in \{(3, 13), (3, 17), (3, 18), (3, 25), (4, 17), (6, 17), (6, 25),$ $(6, 33), (7, 19), (7, 25), (7, 26), (9, 25), (10, 25), (10, 26), (10, 37), (14, 33)\}$.
\end{enumerate}
\end{theorem}

We extend these results and describe improper coloring of $6$-regular toroidal graphs. Based on Theorem~\ref{YehZhu} it is sufficient to explore improper coloring of exceptions.
We analyse each group separately.

\begin{enumerate}

%1
\item Graphs of this section are $(0,0,0,1)$-colorable, which is easy to check due to their small sizes (see Figure~\ref{fig:SmallGraphs}). 

\begin{figure}[h!]
				$$
					\includegraphics[scale=1]{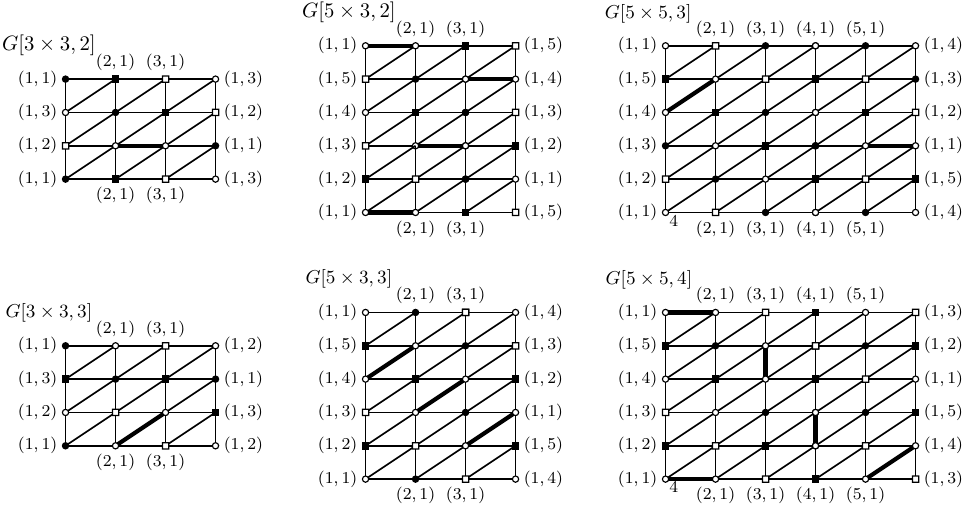}
				$$
				\caption{$(0,0,0,1)$-colorings of six small graphs. Empty circle corresponds to color 4. Bold edge corresponds to a monochromatic edge.}
				\label{fig:SmallGraphs}
\end{figure}	

%2
\item If $G = G[m \times 2, 1]$ where $m$ is odd, then $G$ is $5$-regular graph (ignoring edge multiplicity) and hence it does not belong to desired class.

%3
\item Suppose that $G = G_n[1, r, r+1]$.\\
If $n = 2r+2$, $G_n$ is not $6$-regular graph (again ignoring edge multiplicity).\\
We will show that other cases can be transformed to the case~4.\\
Let $n = 2r+3$.
It holds that $G_{2r+3}[1,r,r+1] \cong G_{2r+3}[1,2,3]$ using unit $(r+1) \mod (2r+3)$ (since $[1,r,r+1] = (r+1) [1,2,3]$).\\
If $n = 3r+1$, it holds that $G_{3r+1}[1,r,r+1] \cong G_{3r+1}[1,2,3]$ with unit $r$%$-$ we use unit $r \mod (3r+1)$
.\\
If $n = 3r+2$, it holds that $G_{3r+2}[1,r,r+1] \cong G_{3r+2}[1,2,3]$ with unit $(r+1)$% \mod (3r+2)$
.

%4
\item Now let $G = G_n[1,2,3]$. \\
If $n = 7$, then $G \cong K_7$, which is $(0,0,0,3)$-colorable.
If $n = 11$, then $G \cong T_{11}$, which is $(0,0,0,2)$-colorable.\\
Let $n>7, n\neq 11$.\\
If $n\equiv 1 \mod 4$, we color vertices $0, 1, \dots, n-1$ by colors $a,b,c,d$ according to the pattern $(abcd)^{\frac{n-5}{4}}abcdd$.\\
If $n\equiv 2 \mod 4$, we use pattern $(abcd)^{\frac{n-10}{4}}(abcdd)^2$.\\
If $n\equiv 3 \mod 4$, we use pattern $(abcd)^{\frac{n-15}{4}}(abcdd)^3$.\\
It follows that circulant graphs with generating set of the form $S = \{1, 2, 3\}$ (except $K_7$ and $T_{11}$) are $(0,0,0,1)$-colorable.

Moreover, any of these colorings needs at most three monochromatic edges.

%5
\item Finally, let $G = G_n[1, r, r+1]$.\\
If $(r,n) = (3, 13)$, then we color $V$ by pattern $abacdcdbabcdd$. \\
If $(r,n) = (6, 17)$, we use pattern $(abcd)^4d$.
Graphs $G_{17}(1, 3, 4)$ and $G_{17}(1, 4, 5)$ are isomorphic with $G_{17}(1, 6, 7)$ (using units 3 and 5), which means they are also $(0,0,0,1^*)$-colorable. \\
If $(r,n) = (3, 18)$, we use pattern $(ababdcdcd)^2$.\\
If $(r,n) = (7, 19)$, we use pattern $dabcdadbcddbcdadbca$. \\
If $(r,n) = (6,25)$, we use pattern $(abcd)^6d$.
Graphs $G_{25}(1, 3, 4)$ and $G_{25}(1, 7, 8)$ are isomorphic with $G_{25}(1, 6, 7)$ (using units 4 and 7). \\
If $(r,n) = (10,25)$, we use pattern $(abcd)^6d$.
Graph $G_{25}(1, 9, 10)$ is isomorphic with $G_{25}(1, 10, 11)$ (using unit 9). \\
If $(r,n) = (10,26)$, we use pattern $((abcd)^3d)^2$.
Graph $G_{26}(1, 7, 8)$ is isomorphic with $G_{26}(1, 10, 11)$ (using unit 7). \\
If $(r,n) = (6,33)$, we use pattern $(abcd)^8d$.
Graph $G_{33}(1, 14, 15)$ is isomorphic with $G_{33}(1, 6, 7)$ (using unit 14).  \\
If $(r,n) = (10, 37)$, we use pattern $(abcd)^9d$. 

Hence, we can see that all these graphs are $(0,0,0,1)$-colorable.
Moreover, any of these colorings needs at most two monochromatic edges.
\end{enumerate}

Based on Theorem \ref{YehZhu} and previous discussion we have the following statement.

\begin{theorem}
\label{6-reg}
Every 6-regular toroidal graph except $K_7$ and $T_{11}$ is $(0,0,0,1)$-colorable.
\end{theorem}

As a consequence of Theorem~\ref{6-reg} we can state the following:
\begin{theorem}
Every not 5-degenerate toroidal graph $G$ is $(0,0,0,3)$-colorable.
\end{theorem}
\begin{proof}
As $G$ is not 5-degenerate, it contains an induced subgraph $H$ with $\delta(H) \geq 6$. As $H$ is also toroidal, then (as a consequence of Eulers formula) it is a 6-regular triangulation embedded on torus. From Theorem~\ref{6-reg} it follows that $H$ admits a $(0,0,0,3)$-coloring (as $K_7$ is $(0,0,0,3)$- and $T_{11}$ is $(0,0,0,2)$-colorable). This coloring can be extended to a $(0,0,0,3)$-coloring of the whole graph $G$ since an interior of a triangular face of $H$ is planar, and hence 4-colorable (even with pre-colored vertices of the triangle). 
\end{proof}

It means that if there is a toroidal graph that is not $(0,0,0,3)$-colorable, then it has to be 5-degenerate.

\section{Conclusion and open problems}

In this paper we extend known results about improper colorings of toroidal graphs.
It is already known that every toroidal graph is $7$-colorable with defect $0$ and $3$-colorable with defect $2$. The question whether every toroidal graph is $4$-colorable with defect $1$ (proposed in~\cite{CowGod}) remains open.

We focused on improper colorings in which not all defect values are the same. We proved that toroidal graphs are $(0,0,0,0,0,1^*)$-colorable, $(0,0,0,0,2)$-colorable, $(0,0,0,$ $1^*,1^*)$-colorable and $(0,0,0,4)$-colorable in general. 
All those results except the last one are tight.

Since $K_7$ is not $(0,0,0,2)$-colorable, we strongly believe that the following holds:

\begin{conjecture}
Every toroidal graph is $(0,0,0,3)$-colorable.
\end{conjecture}

While attempting to lower the constant $4$ in the before mentioned result, we proved that toroidal graphs are $(0,1,2,2)$-colorable.
However, we do not know if this result is optimal in the sense that lowering one of the values will not affect the other ones.
%zamyslenie ci nevieme rychlo a=3, pozriet v partitionoch
\begin{question}
What are optimal $a,b$, such that every toroidal graph is $(0,1,a,b)$-colorable?
\end{question}
%\begin{question}
%What is optimal $a$, such that every toroidal graph is $(0,1,1,a)$-colorable?
%\end{question}
%zamysliet sa co vieme o tych a, b zhora povedat
\begin{question}
What are optimal $a,b$, such that every toroidal graph is $(0,0,a,b)$-colorable?
\end{question}

%In the further research, it is also possible 
Possible direction for the further investigation is to study improper colorings of graphs embeddable on other surfaces. 
For example, graphs embeddable on $\mathbb{N}_1$ (projective planar graphs) are properly $6$-colorable (see~\cite{Hea}). Using the shortest non-contractible cycle we can easily show that such graphs are $(0,0,0,0,1)$- and $(0,0,0,2)$-colorable.
%If we decrease the number of colors, at least one of them needs to have defect value greater than $0$.
%With approach similar to the one used on $4$-colorability of toroidal graphs, we can show these graphs are $(0,0,0,0,1)$-colorable.

In Table~\ref{tab} we sum up already known results about improper colorings (including those proved in this paper).
Rows correspond to surfaces on which graphs are embeddable, in columns we give the results based on the number of colors used in particular improper coloring.
%In the columns we give number of colors used in particular improper coloring.
By $\dagger$ we denote the results that do not need to be optimal.
Recall that planar graphs are not $(k_1, k_2)$-colorable for any positive integers $k_1, k_2$~\cite{Cow} (therefore the same holds for graphs embeddable on surfaces with higher Euler genus).
%Note that there is no column for two colors, since planar graphs are not $(k_1, k_2)$-colorable for any positive integers $k_1, k_2$~\cite{Cow} (therefore the same holds for graphs embeddable on surfaces with higher Euler genus).
%In the table we also mention results for graphs embeddable on $\mathbb{N}_2$, without separate proofs.

\begin{table}[h!]
\centering
\fontsize{8}{11.0}\selectfont{
\begin{tabular}{|>{\centering}m{0.5in}||>{\centering}m{0.85in}|>{\centering}m{0.85in}|>{\centering}m{0.85in}|>{\centering}m{0.98in}|c|}
    \hline
    & $k=7$ & $k=6$ & $k=5$ & $k=4$ & $k=3$  \\
    \hline
    $\mathbb{S}_0$ & $-$ & $-$ & $-$ & $(0,0,0,0)$~\cite{AppHak} & $(2,2,2)$~\cite{Cow}  \\
    \hline
    $\mathbb{S}_1$ & $(0,0,0,0,0,0,0)$ & $(0,0,0,0,0,1^*)$ & \shortstack{\\{$(0,0,0,0,2)$}\\{$(0,0,0,1^*,1^*)$}} & \shortstack{\\{$(0,0,0,4)^\dagger$}\\{$(0,1,2,2)^\dagger$}} & $(2,2,2)$~\cite{CowGodJes}  \\
    \hline	
	\shortstack{{$\mathcal{S}$} \\ {${\rm eg}(\mathcal{S})>2$}} & & &  & $(0,0,0,9{\rm eg}-4)^\dagger$~\cite{ChoEsp} & $(2,2,9{\rm eg}-4)^\dagger$~\cite{ChoEsp} \\
    \hline	
\end{tabular}
}
\caption{Defect lists for improper $k$-colorings of graphs embeddable on $\mathcal{S}$.}
\label{tab}
%\end{center}
\end{table}

\paragraph{Acknowledgement.} 
Authors were supported by the Slovak Research and Development Agency under the contracts No. APVV--19--0153 and APVV--23--0191.
D.~\v{S}vecov\'{a} was also supported by VVGS-2025-3512.

%%%%%%%%%%%%%%%%%%%%%%%%%%%%%%%%%%%%%%%%%%%%%%%%%%%%%%%%%%%%%%%%%%%%%%%%%%%%%
\bibliographystyle{plain}
{
	\bibliography{References}
}

\end{document}